\documentclass{siamltex}
\usepackage{amsmath}
\usepackage{amsfonts}
\usepackage{sw20siam}
\usepackage{graphicx}

\newtheorem{algorithm}[theorem]{Algorithm}

\newtheorem{remark}[theorem]{Remark}

\input{tcilatex}
\begin{document}

\title{Algorithms and Models for Turbulence Not at Statistical Equilibrium}
\author{Nan Jiang\thanks{njiang@fsu.edu, Department of
Scientific Computing, Florida State University, Tallahassee, FL 32306. }
\and William Layton\thanks{%
wjl@pitt.edu, http://www.math.pitt.edu/\symbol{126}wjl, Department of
Mathematics, University of Pittsburgh, Pittsburgh, PA 15260, USA. Partially
supported by NSF grant DMS 1216465 and AFOSR grant FA9550-12-1-0191.}}
\date{11 October 1999}
\maketitle

\begin{abstract}
Standard eddy viscosity models, while robust, cannot represent backscatter
and have severe difficulties with complex turbulence not at statistical
equilibrium. This report gives a new derivation of eddy viscosity models
from an equation for the evolution of variance in a turbulent flow. The new
derivation also shows how to correct eddy viscosity models. The report
proves the corrected models preserve important features of the true Reynolds
stresses. It gives algorithms for their discretization including a minimally
invasive modular step to adapt an eddy viscosity code to the extended
models. A numerical test is given with the usual and over diffusive
Smagorinsky model. The correction (scaled by $10^{-8}$ ) does successfully
exhibit intermittent backscatter.
\end{abstract}

\keyphrases{eddy viscosity, backscatter, complex turbulence}


\section{Introduction}

Eddy viscosity models are the workhorses of practical turbulent flow
simulations, \cite{DP11}. Due to the wide experience with them, their
limitations are also well recognized. They cannot represent backscatter
(intermittent energy flow from turbulent fluctuations back to the mean
velocity) without ad hoc fixes (called "absurdities" in \cite{MY07}) like
negative viscosities. This report shows how to correct eddy viscosity models
systematically to include backscatter based on a new and fundamental
derivation of eddy viscosity models.

To begin, given an ensemble of initial conditions%
\begin{equation*}
u(x,0;\omega _{j})=u_{0}(x;\omega _{j}),j=1,\cdot \cdot \cdot ,J,\text{ }%
x\in \Omega ,
\end{equation*}%
let $u(x,t;\omega _{j}),p(x,t;\omega _{j})$\ be associated solutions to the
Navier-Stokes equations (NSE)%
\begin{eqnarray}
u_{t}+u\cdot \nabla u-\nu \triangle u+\nabla p &=&f(x,t)\text{, and }\nabla
\cdot u=0\text{, in }\Omega , \\
u &=&0\text{ on }\partial \Omega \text{.}  \notag
\end{eqnarray}%
Let $\left\langle \cdot \right\rangle $ denote ensemble averaging%
\begin{equation*}
\left\langle u\right\rangle (x,t):=\frac{1}{J}\sum_{j=1}^{J}u(x,t;\omega
_{j})\text{ and }u^{\prime }(x,t;\omega _{j})=u(x,t;\omega
_{j})-\left\langle u\right\rangle (x,t).
\end{equation*}%
Ensemble averaging the NSE yields the non-closed system: $\nabla \cdot
\left\langle u\right\rangle =0$\ and%
\begin{equation}
\left\langle u\right\rangle _{t}+\left\langle u\right\rangle \cdot \nabla
\left\langle u\right\rangle -\nu \triangle \left\langle u\right\rangle
-\nabla \cdot R(u,u)+\nabla \left\langle p\right\rangle =f(x,t)\text{,}
\label{eq:RaNSe}
\end{equation}%
where the Reynolds stress $R(u,u)$ is%
\begin{equation*}
R(u,u):=\left\langle u\right\rangle \otimes \left\langle u\right\rangle
-\left\langle u\otimes u\right\rangle =-\left\langle u^{\prime }\otimes
u^{\prime }\right\rangle ,
\end{equation*}%
e.g., \cite{BF10}, \cite{DP11}, \cite{MY07}. Statistical models of
turbulence begin with ensemble averaging and replace $R(u,u)$ by an enhanced
viscous term depending only on the mean velocity. We show in Section 2 that
these eddy viscosity models are based on three steps.

1. The Boussinesq assumption (from \cite{B77}, \cite{S43}, proven in \cite%
{L14}) that turbulent fluctuations (the action of $\nabla \cdot R(u,u)$\ in (%
\ref{eq:RaNSe})) are dissipative on average in (\ref{eq:RaNSe}). This is
followed by assuming that space and time \textit{averaged} dissipativity
holds \textit{pointwise} in time and space.

2. The eddy viscosity hypothesis that this dissipativity aligns with the
gradient or deformation tensor and thus can be represented by a viscous term
with a turbulent viscosity coefficient $\nu _{T}(\left\langle u\right\rangle
)$, \cite{R}.

3. Model parametrization/calibration is done by fitting the turbulent
viscosity coefficient $\nu _{T}(\left\langle u\right\rangle )$ to flow data.
Calibration is equivalent to specifying a fluctuation model for $\nabla
u^{\prime }$ in terms of $\nabla \left\langle u\right\rangle $.

The resulting eddy viscosity model (whose solution $w(x,t)$, $q(x,t)$ is
intended to be an approximation of the true flow averages $\left\langle
u\right\rangle ,\left\langle p\right\rangle $) results: $\nabla \cdot w=0$
and 
\begin{gather}
w_{t}+w\cdot \nabla w-\nabla \cdot \left( \left[ \nu +\nu _{T}(w)\right]
\nabla w\right) +\nabla q=f(x,t)\text{, in }\Omega ,  \tag{EV} \\
w=0\text{ on }\partial \Omega \text{\ and }w(x,0)=\left\langle
u_{0}\right\rangle \text{ in }\Omega \text{.}  \notag
\end{gather}%
Eddy viscosity models, with increasingly complex equations determining $\nu
_{T}(w)$, are the models of choice for most industrial turbulent flows, \cite%
{DP11}, and many parameterizations of eddy viscosity models are known, e.g., 
\cite{BIL06}, \cite{CL} \cite{HKJ00}, \cite{IL98}, \cite{LL02}, \cite{LRT12}%
, \cite{MP}, \cite{V04}. They have well recognized limitations in not
modeling complex turbulence, backscatter or turbulence not at statistical
equilibrium, e.g., \cite{FD02}, \cite{LN92}, \cite{S07}, \cite{Starr}. (The
second assumption that the dissipativity of the Reynolds stress term aligns
with $\nabla \left\langle u\right\rangle $ also fails for some flows, \cite%
{LN92}, but is not the issue addressed herein.) 

The correction required for eddy viscosity models\ to represent backscatter
in non-statistically stationary turbulence, the case when the action of the
fluctuations is intermittently non-dissipative, is derived and analyzed
herein. Given the eddy viscosity parameterization $\nu _{T}(w),$ choose a
re-scaling parameter $\beta >0$ and define%
\begin{equation*}
a(w):=\sqrt{\nu ^{-1}\nu _{T}(w)}.
\end{equation*}%
The corrected EV model (derived in Section 2) is then $\nabla \cdot w=0$ and%
\begin{gather}
w_{t}+\beta ^{2}a(w)\frac{\partial }{\partial t}\left( a(w)w\right) +w\cdot
\nabla w  \tag{Corrected EV} \\
-\nabla \cdot \left( \left[ \nu +\nu _{T}(w)\right] \nabla w\right) +\nabla
q=f(x,t)\text{.}  \notag
\end{gather}%
In Section 3 time averaged dissipativity, an important feature of the true\
Reynolds stresses, is proven to be preserved in (Corrected EV).

The (Corrected EV) differs from (EV) by the extra term $\beta ^{2}a(w)\frac{%
\partial }{\partial t}\left( a(w)w\right) $. This term means time
discretization introduces new issues, especially in adapting legacy codes
from (EV) to (Corrected EV). Section 4 shows how time discretization can be
done and preserve these important model properties, including the important
case of modular adaptation of a legacy code available for (EV).
Phenomenology is used to obtain some insight into calibration of the
re-scaling parameter $\beta $\ in Section 5. Section 6 tests correction of
the over-diffused Smagorinsky model. Even with a very small rescaling of the
correction, $\beta ^{2}\simeq O(10^{-8})$, the numerical test shows that the
corrected model does exhibit backscatter.

\section{Derivation of Corrected EV Models}

Beginning with two results from \cite{L14}, this section shows that eddy
viscosity models are based on three assumptions and that they are only
consistent in a time averaged sense. Next we show how to take a given eddy
viscosity model and extend it to be energy consistent with the NSE pointwise
in time.

The ensemble averaged Navier Stokes equations, (\ref{eq:RaNSe}) above,
involve the non-closed Reynolds stress $R(u,u)=-\left\langle u^{\prime
}\otimes u^{\prime }\right\rangle $. This term, which must be modelled,
accounts for the effects of the fluctuations on the mean flow, e.g., \cite%
{CL}, \cite{DP11}, \cite{MP}, \cite{R}, \cite{S01}. Let $||\cdot ||,(\cdot
,\cdot )$ denote the usual $L^{2}(\Omega )$ norm and inner product. Taking
the inner product of (\ref{eq:RaNSe}) with $\left\langle u\right\rangle $
and performing the usual steps gives the equation for the kinetic energy
evolution of $\left\langle u\right\rangle $:%
\begin{equation*}
\frac{d}{dt}\frac{1}{2}||\left\langle u\right\rangle ||^{2}+\nu ||\nabla
\left\langle u\right\rangle ||^{2}+\int_{\Omega }R(u,u):\nabla \left\langle
u\right\rangle dx=(f,\left\langle u\right\rangle ).
\end{equation*}%
Thus, the effect of fluctuations on the mean flow is determined by the sign
of%
\begin{equation*}
RS(t):=\int_{\Omega }R(u,u):\nabla \left\langle u\right\rangle dx.
\end{equation*}%
When $RS(t)>0$, the effect of $R(u,u)$ is dissipative while when $RS(t)<0$,
(volume averaged) backscatter occurs and fluctuations increase the energy in
the mean flow. In \cite{L14}, \cite{JKL14}, two key properties of this
Reynolds stress term were proven: time averaged dissipativity and an
equation for the evolution of variance of fluctuations: 
\begin{gather}
LIM_{T\rightarrow \infty }\frac{1}{T}\int_{0}^{T}RS(t)dt=LIM_{T\rightarrow
\infty }\frac{1}{T}\int_{0}^{T}\int_{\Omega }\nu \left\langle |\nabla
u^{\prime }|^{2}\right\rangle dxdt\geq 0,  \label{eq:BH} \\
\int_{\Omega }R(u,u):\nabla \left\langle u\right\rangle dx=\frac{1}{2}\frac{d%
}{dt}\int \left\langle u^{\prime }\cdot u^{\prime }\right\rangle dx+\nu \int
\left\langle \nabla u^{\prime }:\nabla u^{\prime }\right\rangle dx.
\label{eq:CovarEq2}
\end{gather}%
Eddy viscosity models then follow from three assumptions. First, the
statistical equilibrium assumption that dissipativity holds approximately at
every instant in time%
\begin{equation}
\int_{\Omega }R(u,u):\nabla \left\langle u\right\rangle dx\simeq
\int_{\Omega }\nu \left\langle |\nabla u^{\prime }|^{2}\right\rangle dx.
\label{eq:StatSteadyState}
\end{equation}%
The second is that $\nabla u^{\prime }$\ aligns with $\nabla \left\langle
u\right\rangle $. Third, calibration\footnote{%
Alternately, the dissipation in (2.3) is $\left\langle \nu |\nabla u^{\prime
}|^{2}\right\rangle $ is replaced in the model by $\nu _{T}(\left\langle
u\right\rangle )|\nabla \left\langle u\right\rangle |^{2}$. Ideally, then $%
\left\langle |\nabla u^{\prime }|^{2}\right\rangle \simeq \nu ^{-1}\nu
_{T}(\left\langle u\right\rangle )|\nabla \left\langle u\right\rangle |^{2}$
so $a(\left\langle u\right\rangle )=\sqrt{\nu ^{-1}\nu _{T}(\left\langle
u\right\rangle )}$.} provides a model of the fluctuations in terms of the
mean flow 
\begin{equation}
action(\nabla u^{\prime })\simeq a(\left\langle u\right\rangle )\nabla
\left\langle u\right\rangle .  \label{eq:FluctModel}
\end{equation}%
Thus,%
\begin{gather*}
\int_{\Omega }R(u,u):\nabla \left\langle u\right\rangle dx\simeq
\int_{\Omega }\nu a(\left\langle u\right\rangle )^{2}\nabla \left\langle
u\right\rangle :\nabla \left\langle u\right\rangle dx \\
=\int_{\Omega }-\nabla \cdot \left( \nu a(\left\langle u\right\rangle
)^{2}\nabla \left\langle u\right\rangle \right) \cdot vdx\text{ evaluated at 
}v=\left\langle u\right\rangle .
\end{gather*}%
Letting $\nu _{T}(\left\langle u\right\rangle )=\nu a(\left\langle
u\right\rangle )^{2}$, this yields the eddy \ viscosity closure 
\begin{equation*}
-\nabla \cdot R(u,u)\text{ }\Leftarrow \text{ }-\nabla \cdot \left( \nu
_{T}(\left\langle u\right\rangle )\nabla \left\langle u\right\rangle \right)
+\text{ terms incorporated in }\nabla p\text{.}
\end{equation*}

Far from equilibrium, step (\ref{eq:StatSteadyState}) omits $\frac{1}{2}%
\frac{d}{dt}\int \left\langle u_{j}^{\prime }\cdot u_{j}^{\prime
}\right\rangle dx$ in (\ref{eq:CovarEq2}). This is the term that accounts
for backscatter and other non-equilibrium effects. To model this term, $%
u^{\prime }$ must be expressed in terms of $\left\langle u\right\rangle $.
For this, the simplest (explored herein) is to rescale (by $\beta $, Section
4) the fluctuation model (\ref{eq:FluctModel}), yielding 
\begin{equation*}
action(u^{\prime })\simeq \beta a(\left\langle u\right\rangle )\left\langle
u\right\rangle 
\end{equation*}%
This assumption yields%
\begin{equation*}
\int R(u,u):\nabla \left\langle u\right\rangle dx\simeq \frac{1}{2}\frac{d}{%
dt}\int \beta ^{2}a(\left\langle u\right\rangle )^{2}|\left\langle
u\right\rangle (x,t)|^{2}dx+\int \nu a(\left\langle u\right\rangle
)^{2}|\nabla \left\langle u\right\rangle (x,t)|^{2}dx,
\end{equation*}%
arising from an anisotropic time derivative $\beta ^{2}a(\left\langle
u\right\rangle )\frac{\partial }{\partial t}\left( a\left\langle
u\right\rangle )\left\langle u\right\rangle \right) $ and an eddy viscosity
term $-\nabla \cdot \left( \nu _{T}(\left\langle u\right\rangle )\nabla
\left\langle u\right\rangle \right) $ where $\nu _{T}(\left\langle
u\right\rangle )=\nu a(\left\langle u\right\rangle )^{2}.$ This gives the
closure model%
\begin{equation*}
-\nabla \cdot R(u,u)\simeq \beta ^{2}a(\left\langle u\right\rangle )\frac{%
\partial }{\partial t}\left( a(\left\langle u\right\rangle )\left\langle
u\right\rangle \right) -\nabla \cdot \left( \nu _{T}\left\langle
u\right\rangle )\nabla \left\langle u\right\rangle \right) 
\end{equation*}%
and thus we have the corrected model: $\nabla \cdot w=0$\ and%
\begin{gather}
\ w_{t}+\beta ^{2}a(w)\frac{\partial }{\partial t}\left( a(w)w\right)
+w\cdot \nabla w+\nabla p-\nabla \cdot \left( \lbrack \nu +\nu
_{T}(w)]\nabla w\right) =f\text{ }.
\end{gather}

\section{Analysis of The Corrected Smagorinsky Model}

The classic example of an over-diffused model is the standard Smagorinsky
model for which%
\begin{equation*}
\nu _{T}(w)=\left( C_{s}\delta \right) ^{2}|\nabla w|,\text{ \ where \ }%
C_{s}\simeq 0.1,\delta =\triangle x.
\end{equation*}%
(Other numerical values of $C_{s}$ are also used, \cite{Smag93} Table 1.)
Various fixes for it include van Driest damping (reducing near wall model
dissipation) and Germano's dynamic selection of $\mu =C_{s}^{2}(x,t)$. These
are often successful but the latter leads to negative values of $\mu $\ that
model backscatter but can induce instabilities. Often these are clipped ($%
\mu \Leftarrow \max \{\mu ,0\}$) eliminating backscatter being represented
in the model. 

In this section we prove that the corrected Smagorinsky model preserves the
property of the true turbulent fluctuations that on long time average they
are dissipative, Theorem 3.2. For the Smagorinsky model we have%
\begin{equation*}
\nu _{T}(w)=\left( C_{s}\delta \right) ^{2}|\nabla w|,\text{ and thus }%
a(w)=\nu ^{-1/2}C_{s}\delta \sqrt{|\nabla w|}.
\end{equation*}%
Let $\beta >0$, $a(w)=\nu ^{-1/2}C_{s}\delta \sqrt{|\nabla w|}$ and consider%
\begin{equation}
\left. 
\begin{array}{c}
\ w_{t}+\beta ^{2}a(w)\frac{\partial }{\partial t}\left( a(w)w\right)
+w\cdot \nabla w+\nabla p-\nabla \cdot \left( \lbrack \nu +\nu
_{T}(w)]\nabla w\right) =f, \\ 
\text{ }\nabla \cdot w=0\text{ in }\Omega \times (0,\infty ), \\ 
w=0\text{ on }\partial \Omega ,w(x,0)=w_{0}(x),\text{ in }\Omega .%
\end{array}%
\right\}   \label{eq:ExtendedSmag}
\end{equation}%
The model dissipation function for the effect of the model of the Reynolds
stresses on the kinetic energy in the mean flow will be denoted by%
\begin{equation*}
MD(t):=\int_{\Omega }\beta ^{2}a(w)\frac{\partial }{\partial t}\left(
a(w)w\right) \cdot w+\nu _{T}(w)|\nabla w|^{2}dx.
\end{equation*}%
The numerical tests in Section 6 establish (Figure 6.2) that, even for very
small $\beta $, the extra term does allow the sign of the model dissipation $%
MD(t)$ to fluctuate. We prove the time average of $MD(t)$\ is non-negative. 

We assume%
\begin{gather*}
f(x,t)\in L^{\infty }(0,\infty ;L^{2}(\Omega )),w_{0}(x)\in L^{2}(\Omega ),
\\
w\text{ is a strong solution of (3.1).}
\end{gather*}

\begin{lemma}
For $a(w)=\nu ^{-\frac{1}{2}}C_{s}\delta \sqrt{|\nabla w|}$ we have for all $%
w\in \mathring{W}^{1,3}(\Omega )$%
\begin{equation*}
||a(w)w||^{2}\leq C(\Omega )\nu ^{-1}\left( C_{s}\delta \right) ^{2}||\nabla
w||_{L^{3}}^{3}.
\end{equation*}
\end{lemma}

\begin{proof}
By H\"{o}lder's inequality and the Poincar\'{e}-Friedrichs inequality%
\begin{gather*}
||a(w)w||^{2}=\nu ^{-1}\left( C_{s}\delta \right) ^{2}\int_{\Omega }|\nabla
w||w|^{2}dx \\
\leq \nu ^{-1}\left( C_{s}\delta \right) ^{2}||w||_{L^{3}}^{2}||\nabla
w||_{L^{3}}\leq C(\Omega )\nu ^{-1}\left( C_{s}\delta \right) ^{2}||\nabla
w||_{L^{3}}^{3}.
\end{gather*}
\end{proof}

We prove next that $MD(t)$ dissipates energy in the time average sense.

\begin{theorem}
Let $C_{s}>0,\delta >0,\beta >0$\ and $w$ be a strong solution of (\ref%
{eq:ExtendedSmag}). Then%
\begin{eqnarray}
w\text{ , }a(w)w &\in &L^{\infty }(0,\infty ;L^{2}(\Omega )),
\label{eq:apriori1} \\
\frac{1}{T}\int_{0}^{T}\left( \int_{\Omega }\left[ \nu +\nu _{T}(w\right]
)\nabla w:\nabla wdx\right) dt &\leq &C<\infty ,\text{ }C\text{ independent
of }T\text{.}  \label{eq:a-priori2}
\end{eqnarray}%
The model also satisfies long term balance between energy input and
dissipation: for any generalized limit $LIM$%
\begin{equation*}
\ LIM_{T\rightarrow \infty }\frac{1}{T}\int_{0}^{T}\left( \int_{\Omega }%
\left[ \nu +\nu _{T}(w\right] )|\nabla w|^{2}dx\right) dt=LIM_{T\rightarrow
\infty }\frac{1}{T}\int_{0}^{T}(f,w)dt.
\end{equation*}%
Further%
\begin{equation*}
0\leq \lim_{T\rightarrow \infty }\inf \frac{1}{T}\int_{0}^{T}MD(t)dt\leq
\lim_{T\rightarrow \infty }\sup \frac{1}{T}\int_{0}^{T}MD(t)dt<\infty .
\end{equation*}
\end{theorem}

\begin{proof}
Take the inner product of the corrected Smagorinsky model with $w$. This
yields%
\begin{equation}
\ \frac{1}{2}\frac{d}{dt}\left[ ||w||^{2}+\beta ^{2}||a(w)w||^{2}\right]
+\int_{\Omega }\left[ \nu +\nu _{T}(w)\right] |\nabla w|^{2}dx=(f,w).
\label{eq:EnergyEstforModel}
\end{equation}%
Let%
\begin{equation*}
F:=\frac{1}{2\nu }||f||_{L^{\infty }(0,\infty ;L^{2}(\Omega ))}\text{ \ and
\ }y(t):=\ \frac{1}{2}\left[ ||w||^{2}+\beta ^{2}||a(w)w||^{2}\right] .
\end{equation*}%
By Lemma 3.1,%
\begin{eqnarray*}
2y(t) &=&||w||^{2}+\beta ^{2}||a(w)w||^{2}\leq C\left( \int_{\Omega }\left[
\nu +\nu _{T}(w)\right] |\nabla w|^{2}dx\right)  \\
&\leq&C\left( \nu ||\nabla w||^{2}+\left( C_{s}\delta \right) ^{2}||\nabla
w||_{L^{3}}^{3}\right) .
\end{eqnarray*}
Since $(f,w)\leq \frac{\nu }{2}||\nabla w||^{2}+\frac{1}{2\nu }||f||_{-1}^{2}
$, $y(t)$ thus satisfies%
\begin{equation*}
y^{\prime }(t)+\alpha y(t)\leq F(<\infty )\text{ for some }\alpha >0.
\end{equation*}%
This implies $y(t)\in \ L^{\infty }(0,\infty )$ which is the first claimed 
\'{a} priori bound. This bound now implies that $(f,w)(t)\in L^{\infty
}(0,\infty ).$ Integrating (\ref{eq:EnergyEstforModel}) over $[0,T]$ and
dividing by $T$ gives%
\begin{gather}
\ \frac{1}{2T}\left[ ||w||^{2}+\beta ^{2}||a(w)w||^{2}\right] (T)+\frac{1}{T}%
\int_{0}^{T}\left( \int_{\Omega }\left[ \nu +\nu _{T}(w\right] )|\nabla
w|^{2}dx\right) dt  \label{eq:TimeAveragedEnergyEqn} \\
=\frac{1}{2T}\left[ ||w_{0}||^{2}+\beta ^{2}||a(w_{0})w_{0}||^{2}\right] +%
\frac{1}{T}\int_{0}^{T}(f,w)dt.  \notag
\end{gather}%
Since $(f,w)\in L^{\infty }(0,\infty )$ this implies (\ref{eq:a-priori2})
holds which implies as $T\rightarrow \infty $ \ limit inferiors and
superiors exist. The \'{a} priori estimates (\ref{eq:apriori1}), (\ref%
{eq:a-priori2}) imply that (\ref{eq:TimeAveragedEnergyEqn}) takes the form%
\begin{equation}
\ \mathcal{O}(\frac{1}{T})+\frac{1}{T}\int_{0}^{T}\left( \int_{\Omega }\left[
\nu +\nu _{T}(w\right] )|\nabla w|^{2}dx\right) dt=\mathcal{O}(\frac{1}{T})+%
\frac{1}{T}\int_{0}^{T}(f,w)dt.  \label{eqTimeAveraged2}
\end{equation}%
Letting $T\rightarrow \infty $ implies long term balance between energy
input and dissipation.

Consider now the time average of $MD(t)$. We have%
\begin{eqnarray*}
\frac{1}{T}\int_{0}^{T}MD(t)dt &=&\frac{1}{T}\int_{0}^{T}\int_{\Omega }\left[
\beta ^{2}\frac{\partial }{\partial t}\left( a(w)w\right) \cdot (a(w)w)+\nu
_{T}(w)|\nabla w|^{2}\right] dxdt \\
&=&\frac{1}{T}\int_{0}^{T}\left[ \frac{\beta ^{2}}{2}\frac{d}{dt}%
||a(w)w||^{2}+\left( C_{s}\delta \right) ^{2}||\nabla w||_{L^{3}}^{3}\right]
dt \\
&=&\frac{\beta ^{2}}{2T}\left(
||a(w(T))w(T)||^{2}-||a(w(0))w(0)||^{2}\right) +\frac{1}{T}%
\int_{0}^{T}\left( C_{s}\delta \right) ^{2}||\nabla w||_{L^{3}}^{3}dt.
\end{eqnarray*}%
From (\ref{eq:TimeAveragedEnergyEqn}), the right hand side equals%
\begin{gather*}
\frac{1}{T}\int_{0}^{T}MD(t)dt=\frac{1}{T}\int_{0}^{T}(f,w)dt-\frac{1}{T}%
\int_{0}^{T}\nu ||\nabla w||^{2}dt - \frac{1}{2T} \left[ || w ||^2 (T) -
||w_0||^2\right] ,
\end{gather*}%
while from (\ref{eqTimeAveraged2}) we have%
\begin{eqnarray*}
\frac{1}{T}\int_{0}^{T}MD(t)dt &=&\frac{1}{T}\int_{0}^{T}(f,w)dt-\frac{1}{T}%
\int_{0}^{T}\nu ||\nabla w||^{2}dt \\
&=&\ \mathcal{O}(\frac{1}{T})-\mathcal{O}(\frac{1}{T})+\frac{1}{T}%
\int_{0}^{T}\left( C_{s}\delta \right) ^{2}||\nabla w||_{L^{3}}^{3}dt.
\end{eqnarray*}%
Thus the limit inferior of the time average of $MD(t)$ exists and is
non-negative.
\end{proof}

\section{Time Discretization of Corrected Eddy Viscosity Models}

This section presents three unconditionally stable, linearly implicit time
discretizations of (Corrected EV). Method 1 and the modular Method 2 are
first order accurate. Method 3 is second order accurate. All three preserve
the essential feature of time averaged dissipativity, Proposition 4.3. The
second method shows how a legacy code for solving (EV) can be adapted to
solve (Corrected EV). We suppress the spacial discretization to reduce
notation and focus on essential points. Let the timestep and associated
quantities be denoted as usual by%
\begin{gather*}
\text{timestep}=k,\text{ \ \ }t_{n}=nk,\text{ \ \ }f^{n}(x):=f(x,t_{n}), \\
w^{n}(x)=\text{ approximation to }w(x,t_{n}),\text{ and }a^{n}:=a(w^{n}),%
\text{ \ \ }\nu _{T}^{n}:=\nu _{T}(w^{n}),n\geq 0,
\end{gather*}%
and set $a^{-1}=a^{0}=a(w^{0})$. We consider time discretization of \
(Corrected EV) under no-slip boundary conditions. The first method is: given 
$w^{0},$ find $w^{n}$ satisfying%
\begin{gather}
\frac{w^{n+1}-w^{n}}{k}+\beta ^{2}a^{n}\frac{a^{n}w^{n+1}-a^{n-1}w^{n}}{k} 
\notag \\
+w^{n}\cdot \nabla w^{n+1}-\nabla \cdot \left( \left[ \nu +\nu _{T}^{n}%
\right] \nabla w^{n+1}\right) +\nabla q^{n+1}=f^{n+1}(x)\text{ in }\Omega , 
\tag{Method 1} \\
\nabla \cdot w^{n+1}=0\text{ in }\Omega ,\text{ and }w^{n+1}=0\text{ on }%
\partial \Omega \text{ .}  \notag
\end{gather}%
This method is linearly implicit. We shall prove in Theorem 4.1 that it is
unconditionally, nonlinearly stable. In linearly implicit methods
(considered herein) nonlinearities are lagged to previous time levels (or
extrapolated), so there is no difficulty if $\nu _{T}$ is determined by
solving other nonlinear equations, such as in the $k-\varepsilon $ model, 
\cite{CL}, \cite{MP}.

Comparing the model, this method and the ensemble averaged NSE, we see that

1. True effect of fluctuations on means: $RS(t)=$ $\int_{\Omega
}R(u,u):\nabla \left\langle u\right\rangle dx.$

2. Model: $MD(t)=$ $\int_{\Omega }\beta ^{2}a(w)\frac{\partial }{\partial t}%
\left( a(w)w\right) \cdot w+\nu _{T}(w)|\nabla w|^{2}dx,$

3. Discrete version: $MD^{n+1}=$ $\int_{\Omega }\beta ^{2}a^{n}\frac{%
a^{n}w^{n+1}-a^{n-1}w^{n}}{k}\cdot w^{n+1}+\nu _{T}^{n}|\nabla
w^{n+1}|^{2}dx $.

\begin{theorem}
(Method 1) is unconditionally, nonlinearly energy stable. For any $N\geq 1$%
\begin{gather*}
\left( \frac{1}{2}||w^{N}||^{2}+\frac{1}{2}\beta^2||a^{N-1}w^{N}||^{2}\right)
\\
+\sum_{n=0}^{N-1}\frac{1}{2}\left( ||w^{n+1}-w^{n}||^{2}+\beta
^{2}||a^{n}w^{n+1}-a^{n-1}w^{n}||^{2}\right) \\
+k\sum_{n=0}^{N-1}\int_{\Omega }\left[ \nu +\nu _{T}^{n}\right] |\nabla
w^{n+1}|^{2}dx \\
=\left( \frac{1}{2}||w^{0}||^{2}+\frac{1}{2}\beta^2||a^{-1}w^{0}||^{2}\right)
+k\sum_{n=0}^{N-1}(f^{n+1},w^{n+1}).
\end{gather*}
\end{theorem}

\begin{proof}
Multiply through by $k$, take the $L^{2}$ inner product of (Method 1) with $%
w^{n+1}$ and use $(w^{n}\cdot \nabla w^{n+1},w^{n+1})=0$. This gives 
\begin{gather*}
||w^{n+1}||^{2}-(w^{n+1},w^{n})+\beta^2||a^{n}w^{n+1}||^{2}-%
\beta^2(a^{n-1}w^{n},a^{n}w^{n+1}) \\
+k\int_{\Omega }\left[ \nu +\nu _{T}^{n}\right] |\nabla
w^{n+1}|^{2}dx=k\int_{\Omega }f^{n+1}\cdot w^{n+1}dx.
\end{gather*}%
The second and fourth terms are treated by the polarization identity%
\begin{eqnarray*}
(w^{n+1},w^{n}) &=&\frac{1}{2}||w^{n+1}||^{2}+\frac{1}{2}||w^{n}||^{2}-\frac{%
1}{2}||w^{n+1}-w^{n}||^{2}, \\
(a^{n-1}w^{n},a^{n}w^{n+1}) &=&\frac{1}{2}||a^{n}w^{n+1}||^{2}+\frac{1}{2}%
||a^{n-1}w^{n}||^{2}-\frac{1}{2}||a^{n}w^{n+1}-a^{n-1}w^{n}||^{2}.
\end{eqnarray*}%
Collecting terms and summing from $n=0$ to $N-1$ finishes the proof.
\end{proof}

Since Theorem 4.1 is an energy equality we can identify various effects:

\begin{itemize}
\item Model kinetic energy = $\frac{1}{2}||w^{N}||^{2}+\frac{1}{2}\beta^2
||a^{N-1}w^{N}||^{2}$
\end{itemize}

\begin{itemize}
\item model dissipation = $\int_{\Omega }\nu _{T}^{n}|\nabla w^{n+1}|^{2}dx$
\end{itemize}

\begin{itemize}
\item numerical diffusion = $\frac{1}{2}\left( ||w^{n+1}-w^{n}
||^{2}+\beta^2||a^{n}w^{n+1}-a^{n-1}w^{n}||^{2}\right) .$
\end{itemize}

Rewriting the energy equality in an equivalent form gives%
\begin{gather*}
\frac{1}{2k}\left( ||w^{n+1}||^{2}-||w^{n}||^{2}\right) +\frac{1}{2k}%
||w^{n+1}-w^{n}||^{2}+\nu ||\nabla w^{n+1}||^{2}+ \\
\left\{ \frac{\beta^2}{2k}\left(
||a^{n}w^{n+1}||^{2}-||a^{n-1}w^{n}||^{2}\right) +\frac{\beta^2}{2k}%
||a^{n}w^{n+1}-a^{n-1}w^{n}||^{2}+\int_{\Omega }\nu _{T}^{n}|\nabla
w^{n+1}|^{2}dx\right\} \\
=(f^{n+1},w^{n+1}).
\end{gather*}%
The first and last line arise from the terms in the usual backward Euler
discretization of the NSE. The second line (bracketed) is an equivalent form
of $MD^{n+1}.$

\begin{lemma}
For (Method 1) we have 
\begin{gather*}
MD^{n+1}=\frac{\beta^2}{2k}\left(
||a^{n}w^{n+1}||^{2}-||a^{n-1}w^{n}||^{2}\right) \\
+\frac{\beta^2}{2k}||a^{n}w^{n+1}-a^{n-1}w^{n}||^{2}+\int_{\Omega }\nu
_{T}^{n}|\nabla w^{n+1}|^{2}dx.
\end{gather*}
\end{lemma}

Dissipativity on time average holds for all three methods with the same
manipulations of their discrete energy equality. We record it here for
(Method 1).

\begin{proposition}[Time Averaged Dissipativity]
Suppose $\sup_{0<n<\infty }||f(t^{n})||<\infty $. Then, for $%
T_{N}=N\triangle t,$%
\begin{equation*}
\lim \inf_{T_{N}\rightarrow \infty }\frac{1}{T_{N}}\left( \triangle
t\sum_{n=0}^{N}MD^{n+1}\right) \geq 0.
\end{equation*}
\end{proposition}

\begin{proof}
The proof is a discrete analog of the continuous case and will be omitted.
\end{proof}

\subsection{Modular Correction of EV Models}

Given a code that computes an approximation to the EV model%
\begin{equation}
w_{t}+w\cdot \nabla w-\nabla \cdot \left( \left[ \nu +\nu _{T}(w)\right]
\nabla w\right) +\nabla q=f(x,t)\text{.}  \label{eq:EV model}
\end{equation}%
Algorithm 4.2 presents a minimally intrusive, modular, postprocessor to
solve:%
\begin{equation}
w_{t}+\beta ^{2}a(w)\frac{\partial }{\partial t}\left( a(w)w\right) +w\cdot
\nabla w-\nabla \cdot \left( \left[ \nu +\nu _{T}(w)\right] \nabla w\right)
+\nabla q=f(x,t)\text{.}  \label{eq:ExtEVmdel}
\end{equation}%
Precise stability analysis requires a specific choice of the algorithm used
to solve (\ref{eq:EV model}). For this we select the simple, linearly
implicit, backward Euler method.

\textbf{Derivation and Consistency Error.} To derive the modular
postprocessor, rewrite (\ref{eq:EV model}) and (\ref{eq:ExtEVmdel}) as,
respectively,%
\begin{equation*}
y^{\prime }(t)=f(t,y)\text{ \ and \ }y^{\prime }(t)+\beta ^{2}a(y)\frac{d}{dt%
}(a(y)y)=f(t,y).
\end{equation*}%
The postprocessing given in Step 2 below suffices.

\begin{algorithm}[Method 2]
Given \ $y^{n},y^{n-1},$

\textbf{Step 1: }Calculate $y_{temp}^{n+1}$ by: \  $\frac{%
y_{temp}^{n+1}-y^{n}}{k}=f(t_{n+1},y_{temp}^{n+1}),$

\textbf{Step 2 }: Postprocess to obtain $y^{n+1}$\ from $y_{temp}^{n+1}$\ by:

$\left[ 1+\beta ^{2}a(y^{n})^{2}\right] y^{n+1}=y_{temp}^{n+1}+\beta
^{2}a(y^{n})a(y^{n-1})y^{n}$ .
\end{algorithm}

Eliminating $y_{temp}^{n+1}$ from Step 2 shows that $y^{n+1}$\ satisfies
(with $a^{n}=a(y^{n})$)%
\begin{equation*}
\frac{y^{n+1}-y^{n}}{k}+\beta ^{2}a^{n}\frac{a^{n}y^{n+1}-a^{n-1}y^{n}}{k}%
=f(t_{n+1},y_{temp}^{n+1}).
\end{equation*}%
Although this is close to Method 1, $y_{temp}^{n+1}$ not $y^{n+1}$ occurs in
the RHS. The LHS is clearly a first order approximation to $y^{\prime
}(t)+\beta a(y)(a(y)y)^{\prime }$. The RHS, $f(t_{n+1},y_{temp}^{n+1})$, is
a first order approximation to $f(t,y)$ provided $%
y_{temp}^{n+1}-y^{n+1}=O(k) $. Rearranging Step 2 gives%
\begin{eqnarray*}
y_{temp}^{n+1}-y^{n+1} &=&k\left\{ \beta ^{2}a(y^{n})\frac{%
a(y^{n})y^{n+1}-a(y^{n-1})y^{n}}{k}\right\} \\
&=&k\left( \beta ^{2}a(y)\frac{d}{dt}(a(y)y)\right) +O(k^{2})=O(k).
\end{eqnarray*}%
Thus, Algorithm 4.4 is first order accurate approximation of $y^{\prime
}(t)+\beta ^{2}a(y)(a(y)y)^{\prime }$ = $f(t,y)$.

\textbf{Unconditional Stability of the Modular Algorithm.} The utility of
Algorithm 4.4 thus depends on its stability. This is now analyzed for its
application to the corrected EV model. Algorithm 4.4 for (\ref{eq:ExtEVmdel}%
) reads as follows.

\begin{algorithm}
For $n\geq 0,$ given $w^{n}$

\textbf{Step 1:} Find $w_{temp}^{n+1}$\ satisfying%
\begin{gather*}
\frac{w_{temp}^{n+1}-w^{n}}{k} +w^{n}\cdot \nabla w_{temp}^{n+1} \\
-\nabla \cdot \left( \left[ \nu +\nu _{T}^{n}\right] \nabla
w_{temp}^{n+1}\right) +\nabla q_{temp}^{n+1}=f^{n+1}(x)\text{ in }\Omega , \\
\nabla \cdot w_{temp}^{n+1}=0\text{ in }\Omega ,\text{ and }w_{temp}^{n+1}=0%
\text{ on }\partial \Omega \text{ .}
\end{gather*}

\textbf{Step 2:} Given $w_{temp}^{n+1},q_{temp}^{n+1}$\ find $%
w^{n+1},q^{n+1} $ satisfying%
\begin{gather}
\lbrack 1+\beta ^{2}\left( a^{n}\right) ^{2}]w^{n+1}+\nabla
q^{n+1}=w_{temp}^{n+1}+\beta ^{2}a^{n}a^{n-1}w^{n}\text{ in }\Omega \\
\nabla \cdot w^{n+1}=0\text{ in }\Omega ,\text{ and }w^{n+1}=0\text{ on }%
\partial \Omega \text{ .}  \notag
\end{gather}
\end{algorithm}

We prove Algorithm 4.5 is unconditionally stable.

\begin{theorem}
Algorithm 4.5 is unconditionally, nonlinearly energy stable. For any $N\geq
1 $%
\begin{gather*}
\left( \frac{1}{2}||w^{N}||^{2}+\frac{1}{2}\beta
^{2}||a^{N-1}w^{N}||^{2}\right) \\
+\sum_{n=0}^{N-1}\frac{1}{2}\left(
||w^{n+1}-w_{temp}^{n+1}||^{2}+||w_{temp}^{n+1}-w^{n}||^{2}\right) \\
+k\sum_{n=0}^{N-1}\left( \int_{\Omega }\left[ \nu +\nu _{T}^{n}\right]
|\nabla w_{temp}^{n+1}|^{2}dx\right) \\
=\left( \frac{1}{2}||w^{0}||^{2}+\frac{1}{2}\beta
^{2}||a^{-1}w^{0}||^{2}\right) +k\sum_{n=0}^{N-1}(f^{n+1},w_{temp}^{n+1}).
\end{gather*}
\end{theorem}

\begin{proof}
Take the $L^{2}$ inner product of Step 1 with $w_{temp}^{n+1}$\ and follow
the proof of Theorem 1. This gives%
\begin{gather}
\frac{1}{2}||w_{temp}^{n+1}||^{2}-\frac{1}{2}||w^{n}||^{2}+\frac{1}{2}%
||w_{temp}^{n+1}-w^{n}||^{2}  \tag{Step 1 Energy} \\
+k\int_{\Omega }\left[ \nu +\nu _{T}^{n}\right] |\nabla
w_{temp}^{n+1}|^{2}dx=k(f^{n+1},w_{temp}^{n+1}).  \notag
\end{gather}%
Consider Step 2. Taking the $L^{2}$ inner product with $w^{n+1}$ gives%
\begin{equation*}
||w^{n+1}||^{2}+\beta
^{2}||a^{n}w^{n+1}||^{2}=(w_{temp}^{n+1},w^{n+1})+\beta
^{2}(a^{n-1}w^{n},a^{n}w^{n+1}).
\end{equation*}%
The two terms on the RHS are treated with the polarization identity%
\begin{eqnarray*}
(w_{temp}^{n+1},w^{n+1}) &=&\frac{1}{2}||w_{temp}^{n+1}||^{2}+\frac{1}{2}%
||w^{n+1}||^{2}-\frac{1}{2}||w_{temp}^{n+1}-w^{n+1}||^{2}, \\
(a^{n-1}w^{n},a^{n}w^{n+1}) &=&\frac{1}{2}||a^{n}w^{n+1}||^{2}+\frac{1}{2}%
||a^{n-1}w^{n}||^{2}-\frac{1}{2}||a^{n}w^{n+1}-a^{n-1}w^{n}||,
\end{eqnarray*}%
giving%
\begin{gather*}
\frac{1}{2}||w^{n+1}||^{2}+\frac{1}{2}\beta ^{2}||a^{n}w^{n+1}||^{2}-\frac{%
\beta ^{2}}{2}||a^{n-1}w^{n}||^{2}+\frac{1}{2}||w_{temp}^{n+1}-w^{n+1}||^{2}
\\
+\frac{\beta ^{2}}{2}||a^{n}w^{n+1}-a^{n-1}w^{n}||=\frac{1}{2}%
||w_{temp}^{n+1}||^{2}.
\end{gather*}%
Insert the LHS for $\frac{1}{2}||w_{temp}^{n+1}||^{2}$ in (Step 1 Energy).
This gives%
\begin{gather*}
\frac{1}{2}||w^{n+1}||^{2}-\frac{1}{2}||w^{n}||^{2}+\frac{1}{2}\beta
^{2}||a^{n}w^{n+1}||^{2}-\frac{1}{2}\beta ^{2}||a^{n-1}w^{n}||^{2} \\
+\frac{1}{2}||w_{temp}^{n+1}-w^{n}||^{2}+\frac{1}{2}%
||w_{temp}^{n+1}-w^{n+1}||^{2}+\frac{\beta ^{2}}{2}%
||a^{n}w^{n+1}-a^{n-1}w^{n}||^{2} \\
+k\int_{\Omega }\left[ \nu +\nu _{T}^{n}\right] |\nabla
w_{temp}^{n+1}|^{2}dx=k\int_{\Omega }f^{n+1}\cdot w_{temp}^{n+1}dx.
\end{gather*}%
The result now follows by summing over $n$.
\end{proof}

\textbf{A Second Order Time Discretization. }We present a second order
method comprised of an IMEX combination of BDF2 and AB2 adapted to the new
kinetic energy term. It shortens the notation considerably to denote the
linear extrapolation of a variable $\phi $ to $t^{n+1}$ by $\phi ^{\ast
n+1}: $%
\begin{equation*}
\phi ^{\ast n+1}:=2\phi ^{n}-\phi ^{n-1},\text{ for }\phi =w,a,\nu _{T}.
\end{equation*}
The method is: given $w^{0}$, $w^{1}$, $w^{2}$ and $w^{3}$ (found by another
method) find $w^{n+1}$ \ for $n\geq 3$\ satisfying%
\begin{gather}
\frac{3w^{n+1}-4w^{n}+w^{n-1}}{2k}  \tag{Method 3} \\
+\beta ^{2}a^{\ast n+1}\frac{3a^{\ast n+1}w^{n+1}-4a^{\ast n}w^{n}+a^{\ast
n-1}w^{n-1}}{2k}  \notag \\
+w^{\ast n+1}\cdot \nabla w^{n+1}-\nabla \cdot \left( \left[ \nu +\nu
_{T}^{\ast n+1}\right] \nabla w^{n+1}\right) +\nabla q^{n+1}=f^{n+1}(x)\quad 
\text{in }\Omega ,  \notag \\
\nabla \cdot w^{n+1}=0\text{ in }\Omega ,\text{ and }w^{n+1}=0\text{ on }%
\partial \Omega \text{ .}  \notag
\end{gather}

\begin{theorem}
The method (Method 3) is unconditionally, nonlinearly, long time stable. For
any $N\geq 3$%
\begin{gather*}
\frac{1}{4}\big(\Vert w^{N}\Vert ^{2}+\Vert w^{\ast N+1}\Vert ^{2}\big) \\
+\frac{1}{4}\Big(\beta ^{2}\Vert a^{\ast N}w^{N}\Vert ^{2}+\beta ^{2}\Vert
2a^{\ast N}w^{N}-a^{\ast N-1}w^{N-1}\Vert ^{2}\Big) \\
+\sum_{n=3}^{N-1}\Big \{\frac{1}{4}\beta ^{2}\Vert a^{\ast
n+1}w^{n+1}-2a^{\ast n}w^{n}+a^{\ast n-1}w^{n-1}\Vert ^{2} \\
+\frac{1}{4}\Vert w^{n+1}-2w^{n}+w^{n-1}\Vert ^{2}+k\int_{\Omega }\left[ \nu
+\nu _{T}^{\ast n+1}\right] |\nabla w^{n+1}|^{2}dx\Big \} \\
=\sum_{n=3}^{N-1}k(f^{n+1},w^{n+1})+\frac{1}{4}\big(\Vert w^{3}\Vert
^{2}+\Vert 2w^{3}-w^{2}\Vert ^{2}\big) \\
+\frac{1}{4}\Big(\beta ^{2}\Vert a^{\ast 3}w^{3}\Vert ^{2}+\beta ^{2}\Vert
2a^{\ast 3}w^{3}-a^{\ast 2}w^{2}\Vert ^{2}\Big).
\end{gather*}
\end{theorem}

\begin{proof}
Take the $L^{2}$ inner product of (Method 3) with $w^{n+1}$ and multiply
through by $k$. This gives 
\begin{gather*}
\frac{1}{4}\big(\Vert w^{n+1}\Vert ^{2}+\Vert w^{\ast n+2}\Vert ^{2}\big)-%
\frac{1}{4}\left( \Vert w^{n}\Vert ^{2}+\Vert w^{\ast n+1}\Vert ^{2}\right)
\\
+\frac{1}{4}\Big(\beta ^{2}\Vert a^{\ast n+1}w^{n+1}\Vert ^{2}+\beta
^{2}\Vert 2a^{\ast n+1}w^{n+1}-a^{\ast n}w^{n}\Vert ^{2}\Big) \\
-\frac{1}{4}\Big(\beta ^{2}\Vert a^{\ast n}w^{n}\Vert ^{2}+\beta ^{2}\Vert
2a^{\ast n}w^{n}-a^{\ast n-1}w^{n-1}\Vert ^{2}\Big) \\
+\frac{1}{4}\beta ^{2}\Vert a^{\ast n+1}w^{n+1}-2a^{\ast n}w^{n}+a^{\ast
n-1}w^{n-1}\Vert ^{2} \\
+\frac{1}{4}\Vert w^{n+1}-2w^{n}+w^{n-1}\Vert ^{2}+k\int_{\Omega }\left[ \nu
+\nu _{T}^{\ast n+1}\right] |\nabla w^{n+1}|^{2}dx=k(f^{n+1},w^{n+1}).
\end{gather*}%
Summing up above equality from $n=3$ to $n=N-1$ completes the proof.
\end{proof}

\section{The Re-scaling Parameter $\protect\beta $}

If $\nu _{T}(w)$ is well calibrated, then with $a(w)$ $=\sqrt{\nu
_{T}(w)/\nu }$, we begin with the fluctuation model%
\begin{equation}
action(\nabla u^{\prime })\simeq a(\left\langle u\right\rangle )\nabla
\left\langle u\right\rangle .  \label{eq:FMgradu}
\end{equation}%
A fluctuation model relating $action(u^{\prime })$ to $\left\langle
u\right\rangle $ is also needed. It is plausible but incorrect to begin with 
$action(u^{\prime })\simeq a(\left\langle u\right\rangle )\left\langle
u\right\rangle $. For reasons developed next, this relation must be rescaled
(by a factor $\beta <<1$) yielding%
\begin{equation}
\text{\ }action(u^{\prime })\simeq \beta \cdot a(\left\langle u\right\rangle
)\left\langle u\right\rangle .  \label{eq:FMu}
\end{equation}%
\ 

\begin{definition}
Let $u$ denote an ensemble of realizations of the NSE with perturbed
initial data. The associated turbulent intensities of $u$ and $\nabla u$
are, respectively, 
\begin{equation*}
I(u):=\frac{\left\langle ||u^{\prime }||^{2}\right\rangle }{||\left\langle
u\right\rangle ||^{2}}\text{ \ and \ }I(\nabla u):=\frac{\left\langle
||\nabla u^{\prime }||^{2}\right\rangle }{||\nabla \left\langle
u\right\rangle ||^{2}}.
\end{equation*}
\end{definition}

The development of analytic insight into the scaling parameter $\beta $ is
based on phenomenology \textit{associating statistical means and
fluctuations respectively with large and small spacial scales}. Recall that,
e.g., \cite{BatchelorHomoTurb}, \cite{Frisch}, \cite{MY07}, \cite{Pope}, for
fully developed turbulence, kinetic energy is concentrated in the large
scales ($\left\langle ||u^{\prime }||^{2}\right\rangle <<||\left\langle
u\right\rangle ||^{2}$) while energy dissipation is concentrated in the
small scales ($\left\langle ||\nabla u^{\prime }||^{2}\right\rangle
>>||\nabla \left\langle u\right\rangle ||^{2}$), \cite{DP11}, \cite{Frisch}, 
\cite{Pope}. On the scales where one is significant, the other is
negligible. This implies that for fully developed turbulence%
\begin{equation*}
I(u)<<1<<I(\nabla u).
\end{equation*}%
Beginning with the fluctuation models (\ref{eq:FMgradu}), (\ref{eq:FMu}) we
thus have%
\begin{equation*}
\beta ^{2}\frac{||a(\left\langle u\right\rangle )\left\langle u\right\rangle
||^{2}}{||\left\langle u\right\rangle ||^{2}}\simeq \frac{\left\langle
||u^{\prime }||^{2}\right\rangle }{||\left\langle u\right\rangle ||^{2}}<<1<<%
\frac{\left\langle ||\nabla u^{\prime }||^{2}\right\rangle }{||\nabla
\left\langle u\right\rangle ||^{2}}\simeq \frac{||a(\left\langle
u\right\rangle )\nabla \left\langle u\right\rangle ||^{2}}{||\nabla
\left\langle u\right\rangle ||^{2}}.
\end{equation*}%
The quantities%
\begin{equation*}
\frac{||a(\left\langle u\right\rangle )\left\langle u\right\rangle ||^{2}}{%
||\left\langle u\right\rangle ||^{2}}\text{ \ \& }\frac{||a(\left\langle
u\right\rangle )\nabla \left\langle u\right\rangle ||^{2}}{||\nabla
\left\langle u\right\rangle ||^{2}}
\end{equation*}%
both represent (squares of) weighted averages of $a(\left\langle
u\right\rangle )$. If (as expected) these are of comparable magnitude, $%
\beta <<1$ (as expected) and we obtain the estimate%
\begin{equation}
\beta ^{2}\simeq \frac{I(u)}{I(\nabla u)}<<1.  \label{eq:BetaSmall}
\end{equation}

\textbf{Predictions of Phenomenology.} In $3d$, for fully developed,
homogeneous, isotropic turbulence an estimate of this quotient (and thus $%
\beta $) can be calculated from the K41 theory and its predicted,
time-averaged, energy density distribution $E(k)\simeq \alpha \varepsilon
^{2/3}k^{-5/3}$, \cite{Pope}. Associate means with a length scale (typically
the given mesh-width) $\delta $. The well resolved scales and unresolved
scales are then, respectively, $\pi /L<k<\pi /\delta $ and $\pi /\delta
<k<\pi /\eta .$ $\eta =\func{Re}^{-3/4}L$\ is the Kolmogorov micro-scale. We
then calculate%
\begin{equation*}
I(u)\simeq \frac{\int_{\pi /\delta }^{\pi /\eta }E(k)dk}{\int_{\pi /L}^{\pi
/\delta }E(k)dk}=\frac{\left( \frac{\pi }{\eta }\right) ^{-2/3}-\left( \frac{%
\pi }{\delta }\right) ^{-2/3}}{\left( \frac{\pi }{\delta }\right)
^{-2/3}-\left( \frac{\pi }{L}\right) ^{-2/3}}.
\end{equation*}%
In the limit Re $\rightarrow \infty ,\eta \rightarrow 0$ and $\delta <<L$%
\begin{equation*}
I(u)\rightarrow \frac{\left( \frac{\delta }{L}\right) ^{2/3}}{1-\left( \frac{%
\delta }{L}\right) ^{2/3}}=\left( \frac{\delta }{L}\right) ^{2/3}+H.O.T.s.
\end{equation*}%
Similarly, we calculate for $\delta >>\eta $%
\begin{equation*}
I(\nabla u)\simeq \frac{\int_{\pi /\delta }^{\pi /\eta }k^{2}E(k)dk}{%
\int_{\pi /L}^{\pi /\delta }k^{2}E(k)dk}=\left( \frac{\delta }{\eta }\right)
^{4/3}+H.O.T.s.
\end{equation*}%
In particular, dropping higher order terms, $I(u)\simeq \left( \delta
/L\right) ^{2/3}<<$ $1<<$ $\ I(\nabla u)\simeq \left( \delta /\eta \right)
^{4/3}.$ Thus, to leading order, in $3d$%
\begin{equation*}
\beta \simeq \sqrt{\frac{I(u)}{I(\nabla u)}}\simeq Re^{-1/2}\left( \frac{%
\delta }{L}\right) ^{-2/3}.
\end{equation*}

\begin{remark}[The 2d Case]
Phenomenology of two dimensional turbulence is more complicated. The
simplest case for forced turbulence is when the model incorporates some
extra mechanism to extract energy from the largest scales, energy is
injected in an intermediate scale and $\delta /L$ is smaller than the
injection scale but much larger than the micro-scale. Adapting the above
spectral calculation to $2d$ gives%
\begin{equation*}
\beta \simeq \frac{\delta }{L}(\ln \frac{\delta }{\eta })^{-1/2}.
\end{equation*}
\end{remark}

\textbf{Mesh Dependence.} Apparently one option is to\ calculate the
turbulent intensities on a given mesh and use these to find the re-scaling
parameter $\beta $. Unfortunately, the result is severely limited by the
chosen mesh as we now develop. Suppose spacial discretization is performed
by a standard, conforming finite element method based on a mesh of elements $%
e$ with element diameter (the local mesh width) denoted \ $h_{e}$. For
meshes satisfying an angle condition eliminating nearly degenerate elements
and piecewise polynomial finite element velocities $u_{h}$, the inverse
property 
\begin{equation}
||\nabla u_{h}||_{L^{2}(e)}\leq C_{INV}h_{e}^{-1}||u_{h}||_{L^{2}(e)}
\label{eq:InverseEst}
\end{equation}%
holds, where the constant depends only on local polynomial degree and mesh
geometry (element angles). This implies that%
\begin{equation*}
||\nabla u_{h}||\leq C_{INV}h^{-1}||u_{h}||\text{ \ where \ }h:=min_{e}h_{e}.
\end{equation*}

Let $u_{h}$\ denote an ensemble of discrete velocities computed on the same,
fixed mesh as described above. \ Define the characteristic length-scale $L$
of the ensemble mean velocity \ (expected but not guaranteed to be large) by%
\begin{equation*}
L^{-1}:=\frac{||\nabla \left\langle u\right\rangle ||}{||\left\langle
u\right\rangle ||}.
\end{equation*}

\begin{proposition}
Suppose (\ref{eq:InverseEst}) holds. Then 
\begin{equation*}
\beta _{h}:=\sqrt{\frac{I(u_{h})}{I(\nabla u_{h})}}
\end{equation*}%
satisfies%
\begin{equation*}
C_{INV}\frac{h}{L}\leq \beta _{h}\leq C_{PF}\frac{1}{L}.
\end{equation*}
\end{proposition}

\begin{proof}
We have, by rearranging,%
\begin{equation*}
\frac{I(u_{h})}{I(\nabla u_{h})}=\frac{\left\langle ||u_{h}^{\prime
}||^{2}\right\rangle }{\left\langle ||\nabla u_{h}^{^{\prime
}}||^{2}\right\rangle }\frac{\left\langle ||\nabla \left\langle
u_{h}\right\rangle ||^{2}\right\rangle }{||\left\langle u_{h}\right\rangle
||^{2}}=\frac{\left\langle ||u_{h}^{\prime }||^{2}\right\rangle }{%
\left\langle ||\nabla u_{h}^{^{\prime }}||^{2}\right\rangle }L^{-2}.
\end{equation*}%
From the inverse estimate and the Poincar\'{e}-Friedrichs inequality (noting
that $C_{PF}^{2}$ $\ =O($\textit{diameter of }$\Omega )$) we have%
\begin{equation*}
C_{PF}^{-1}||u_{h}^{\prime }||\leq ||\nabla u_{h}^{\prime }||\leq
C_{INV}h^{-1}||u_{h}^{\prime }||.
\end{equation*}%
Thus,%
\begin{equation*}
C_{PF}^{-2}\frac{\left\langle ||u_{h}^{\prime }||^{2}\right\rangle }{%
\left\langle ||\nabla u_{h}^{^{\prime }}||^{2}\right\rangle }\leq 1\leq
C_{INV}^{2}h^{-2}\frac{\left\langle ||u_{h}^{\prime }||^{2}\right\rangle }{%
\left\langle ||\nabla u_{h}^{^{\prime }}||^{2}\right\rangle }.
\end{equation*}%
Rearranging, it follows from the definition of $L$ that, as claimed,%
\begin{equation*}
C_{INV}^{-2}h^{2}L^{-2}\leq \frac{I(u_{h})}{I(\nabla u_{h})}=\frac{%
\left\langle ||u_{h}^{\prime }||^{2}\right\rangle }{\left\langle ||\nabla
u_{h}^{^{\prime }}||^{2}\right\rangle }L^{-2}\leq C_{PF}^{2}L^{-2}.
\end{equation*}
\end{proof}

\begin{remark}
Phenomenology (\ref{eq:BetaSmall}) suggests that the true\ value\ of $\beta $%
\ is small, $\beta <<1$. On an under refined mesh, a reasonable default
choice of $\beta $ seems to be%
\begin{equation*}
\beta (x)=h_{e}(x)^{2}\text{ or }\beta =\left( min_{e}h_{e}\right) ^{2}.
\end{equation*}
\end{remark}

\section{Numerical Explorations}

The corrected model and its time discretizations give a closure that is
dissipative on time average. The remaining question is whether the
correction incorporates backscatter, i.e., whether $MD(t)$ changes sign
while being nonnegative on time average. To test the theory, we consider the
Smagorinsky model (rather than models which perform better in practical
calculations). This is because the Smagorinsky model is over-diffused and
among models in use likely the one for which backscatter would be most
difficult to introduce. Next, since it is believed that an inverse cascade
and backscatter are much more common in the physics of $2d$ flow at high $Re$
than for $3d$ flows, we have selected a $2d$ test problem. (It is also one
for which we have done a number of detailed simulation of the evolution of
velocity ensembles in \cite{J14}, \cite{JKL14}, \cite{JL14a}, \cite{JL14b}.
While not directly relevant herein, this experience with velocity ensembles
for this flow was useful in validation.)

\begin{figure}[h]
\begin{center}
{\Large 
\begin{tabular}{cc}
&  \\ 
\includegraphics[width=11.1cm,height=7.6cm]{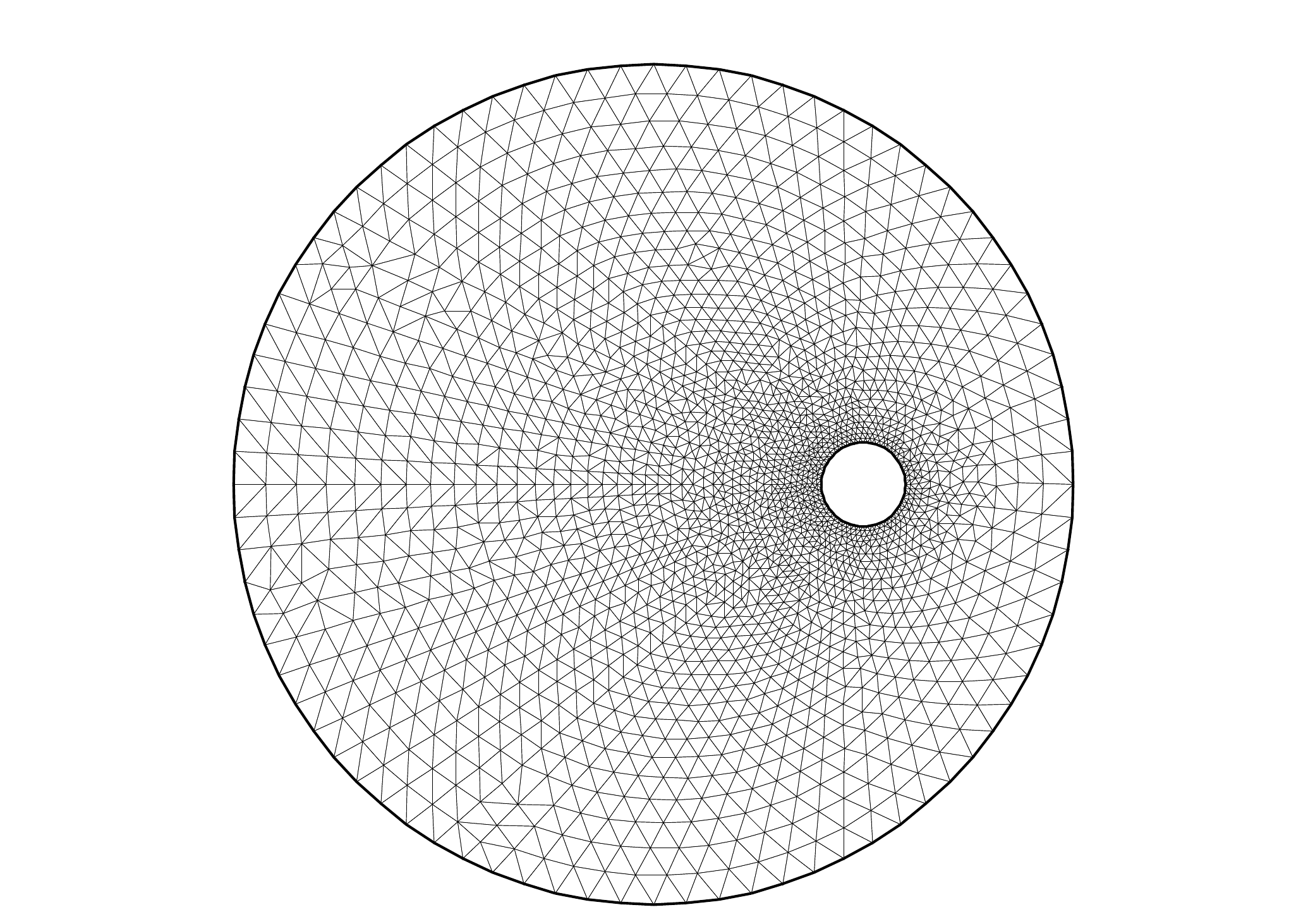} &  \\ 
& 
\end{tabular}
}
\end{center}
\caption{Shown above is the mesh used for the flow between two offest
circles.}
\label{Fig:0}
\end{figure}

\textbf{Test Problem: 2D Flow Between Offset Circles. } Pick%
\begin{gather*}
\Omega =\{(x,y):x^{2}+y^{2}\leq r_{1}^{2}\text{ and }%
(x-c_{1})^{2}+(y-c_{2})^{2}\geq r_{2}^{2}\}, \\
r_{1}=1,r_{2}=0.1,c=(c_{1},c_{2})=(\frac{1}{2},0), \\
f(x,y,t)=(-4y(1-x^{2}-y^{2}),4x(1-x^{2}-y^{2}))^{T},
\end{gather*}%
with no-slip boundary conditions on both circles. The flow (inspired by the
extensive work on variants of Couette flow, \cite{EP00}), driven by a
counterclockwise force (with $f\equiv 0$ at the outer circle), rotates about 
$(0,0)$ and interacts with the immersed circle. This induces a von K\'{a}rm%
\'{a}n vortex street which re-interacts with the immersed circle creating
more complex structures. This flow also exhibits near wall turbulent streaks
and a central polar vortex that pulsates. We discretize in space using the
usual finite element method with Taylor-Hood elements, \cite{G89}, using the
code FreeFEM++, \cite{HP} and in time using (Method 1). For the Smagorinsky
model to simplify notations we have previously used the full gradient in $%
\nu _{T}$. In the implementation, we have used $\tilde{S}_{ij}$, the stain
rate or deformation tensor, 
\begin{equation*}
\nu _{T}=(0.1\Delta x)^{2}|\tilde{S}|.
\end{equation*}%
The choice $C_{s}=0.1$ is common though not universal, see Table 1 in \cite%
{Smag93}. Here $|\tilde{S}|=\sqrt{2\tilde{S}_{ij}\tilde{S}_{ij}}$. All the
theory of the previous sections applies to choosing $S$ instead of $\nabla w$%
. We take $\Delta x$ to be the length of the shortest edge of all triangles.
We also take 
\begin{equation*}
a(\cdot )=\sqrt{\frac{\nu _{T}}{\nu }},\beta =8\ast 10^{-5},\nu
=10^{-4},\triangle t=0.01,T=10.
\end{equation*}%
Here $T$ is the simulation time. The numerical solutions are computed on an
under-resolved, Delaunay-generated triangular mesh with $80$ mesh points on
the outer circle and $60$ mesh points on the inner circle, providing $18,638$
total degrees of freedom, refined near the inner circle (see Figure 6.1).
For this mesh the shortest edge of all triangles is $min_{e}h_{e}=0.0110964$
and the longest edge $max_{e}h_{e}=0.108046$.

We compute the following quantities. 
\begin{align*}
& MD=\int_{\Omega }\beta ^{2}a(t^{n})\left( \frac{%
a(t^{n})w(t^{n+1})-a(t^{n-1})w(t^{n})}{\triangle t}\right) \cdot w(t^{n+1})dx
\\
& \qquad \qquad \qquad +\nu _{T}\Vert \nabla w^{n+1}\Vert ^{2}, \\
& TMD=\int_{\Omega }\beta ^{2}a(t^{n})\left( \frac{%
a(t^{n})w(t^{n+1})-a(t^{n-1})w(t^{n})}{\triangle t}\right) \cdot
w(t^{n+1})dx, \\
& EVD=\int_{\Omega }\nu _{T}^{n}|\nabla w^{n+1}|^{2}dx, \\
& VD=\nu \Vert \nabla w^{n+1}\Vert ^{2}.
\end{align*}%
Note that if $\beta =0$ (i.e., if we were solving the usual Smagorinsky
model) we would have $MD=EVD>0$. Observe in the first plot of Figure 6.2
that with $\beta =8\ast 10^{-5}$, $MD(t)$ is on time average positive
(consistent with theoretical predictions) but there are times when $MD$
becomes negative and indicates backscatter. Thus the corrections to the eddy
viscosity models do have built into them the possibility of representing
backscatter. Various other statistics are also plotted in Figures 6.2
including TMD which represents the effect of the new term, the eddy
viscosity dissipation term EVD and the viscous dissipation term VD.

\begin{figure}[h]
\begin{center}
{\Large 
\begin{tabular}{cc}
\includegraphics[width=12.1cm,height=3.9cm]{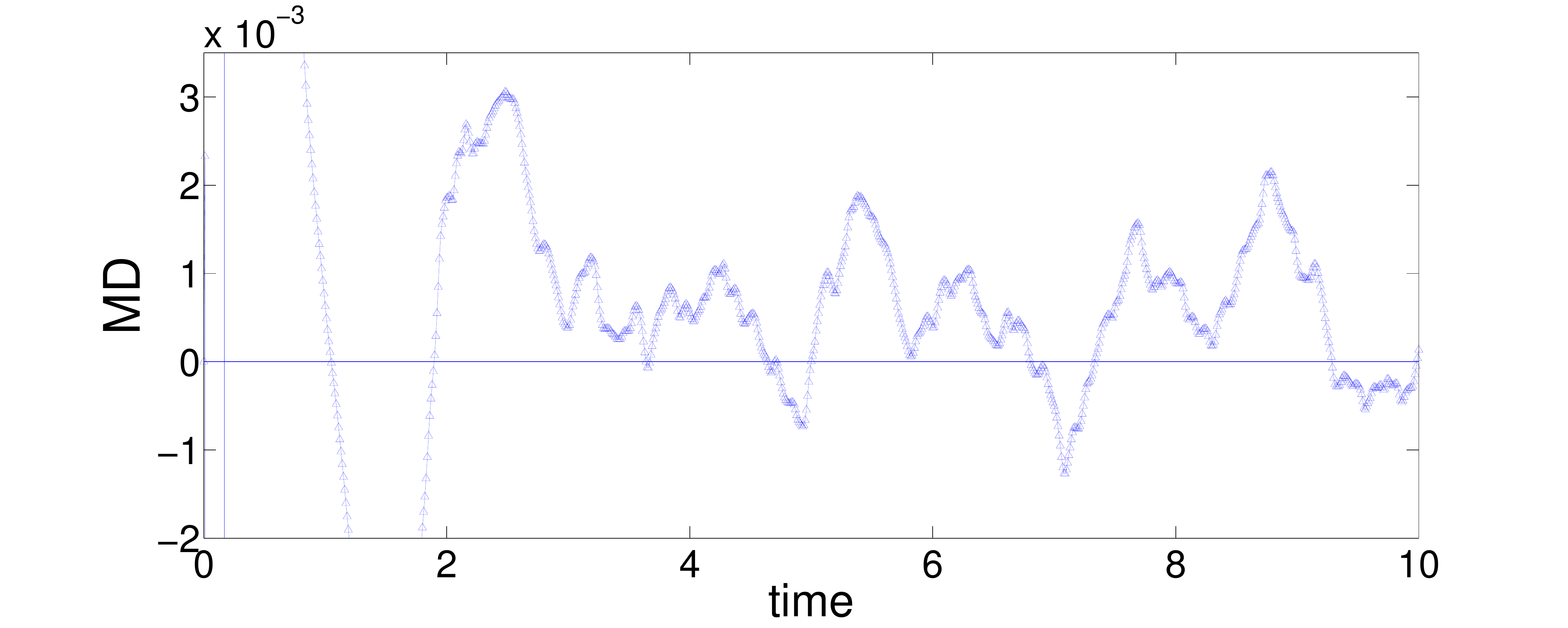} &  \\ 
\includegraphics[width=12.1cm,height=3.9cm]{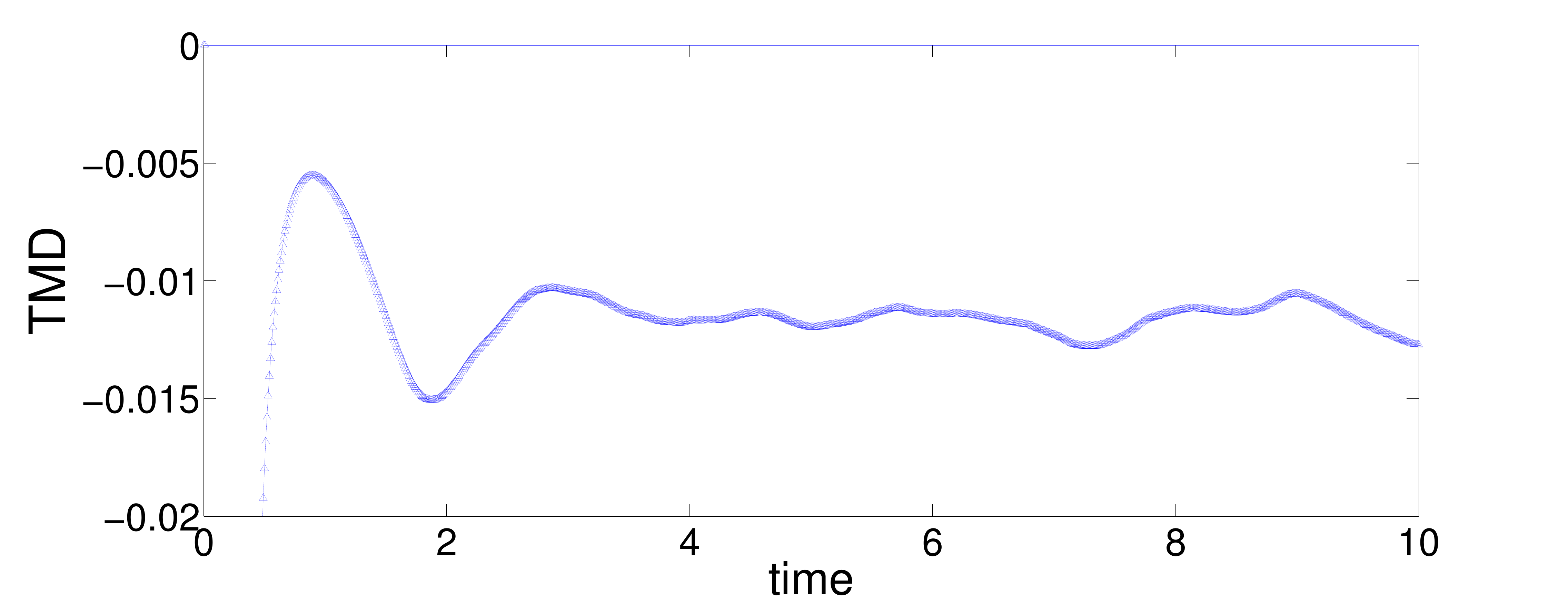} &  \\ 
\includegraphics[width=12.1cm,height=3.9cm]{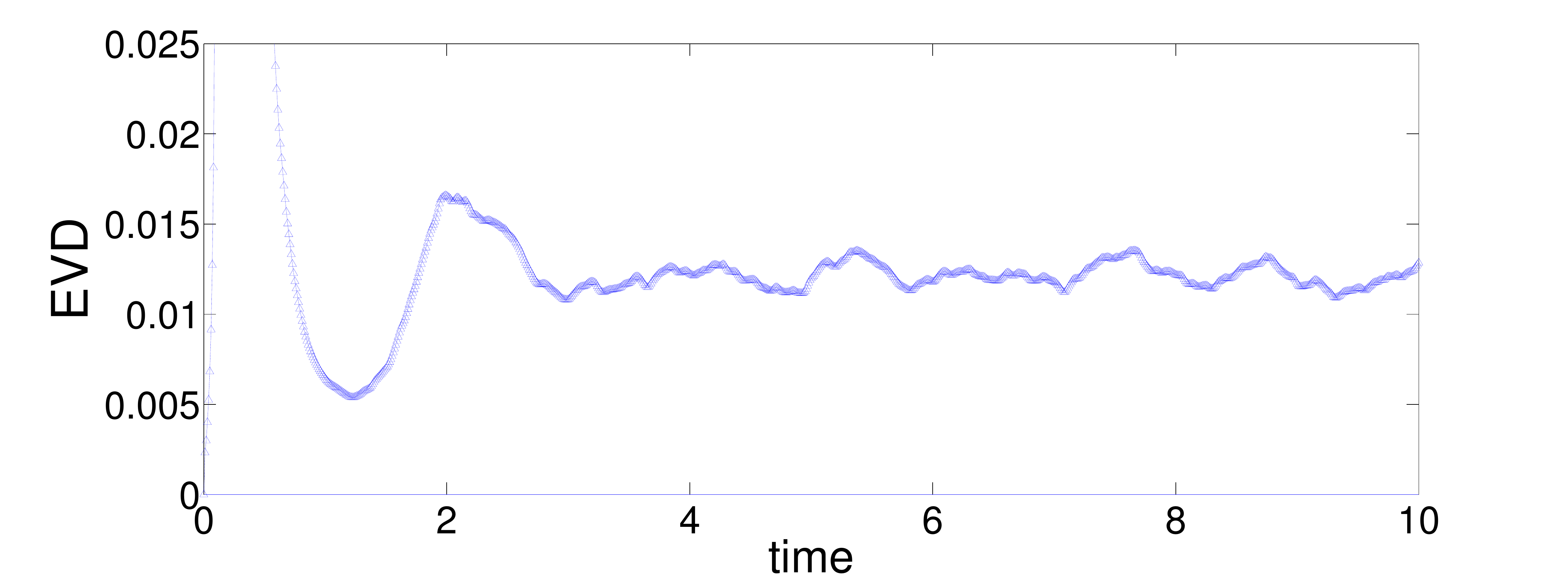} &  \\ 
\includegraphics[width=12.1cm,height=3.9cm]{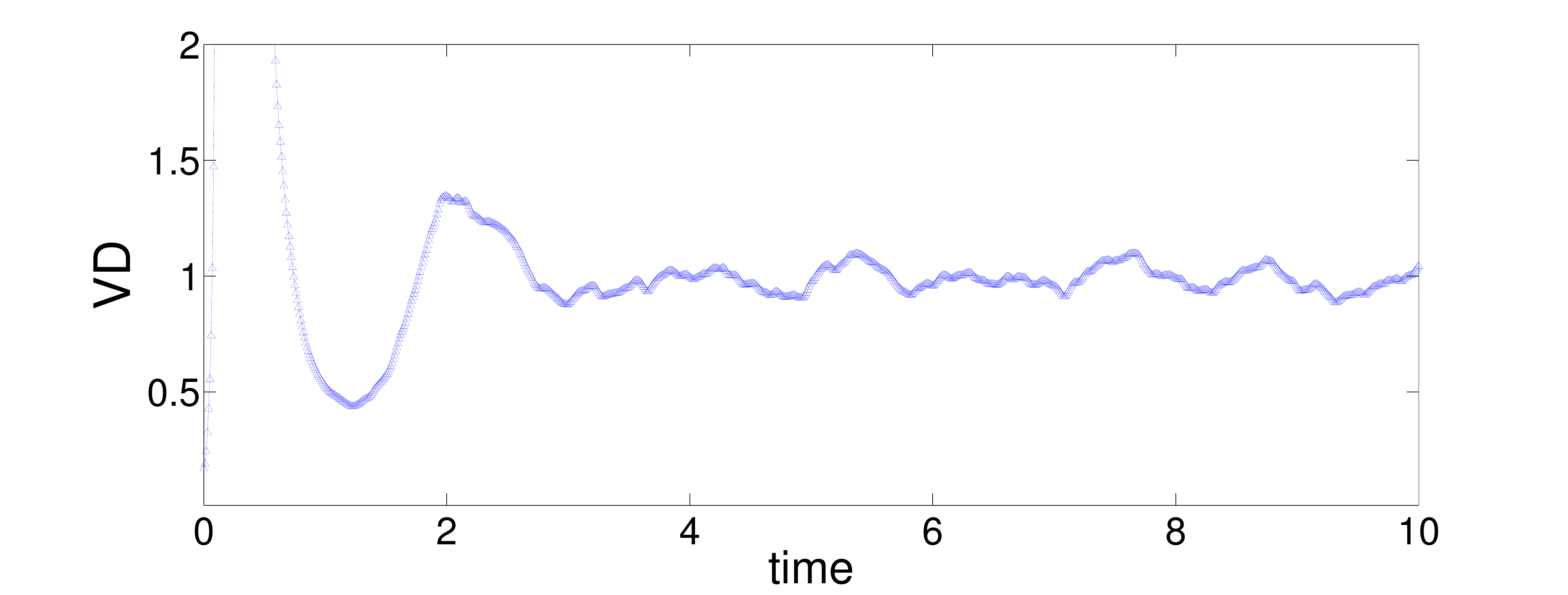} &  \\ 
& 
\end{tabular}
}
\end{center}
\caption{$\protect\nu =1/10000$, $\Delta t=0.01$, $m=80$, $n=60$, $\protect%
\beta =8\ast 10^{-5}$.}
\label{Fig:1}
\end{figure}

\section{Conclusions}

We have shown that eddy viscosity models can be quite easily adapted to
non-equilibrium turbulence and incorporate backscatter without using
negative turbulent viscosities. The modified eddy viscosity model has been
tested successfully for the Smagorinsky model, chosen because it is over
dissipative. Some preliminary and formal calculations were given for the
scaling parameter $\beta $. It may also happen that for accuracy a different
fluctuation model may be needed for the viscous dissipation term and the
kinetic energy term. This is an open question. Strong solutions of the new
models have been proven to share the property of the true Reynolds stresses
of being dissipative on time average. We have also given three methods for
time discretization preserving this property, including a modular correction
for an eddy viscosity code. There are many important open questions
including accuracy tests, existence of weak solutions to the new models,
model calibration and extension to and testing for better EV models.

\end{document}